\documentclass[12pt,a4paper,reqno]{amsart}
\usepackage{amsmath,amssymb,amsthm,bm,wasysym,calc,verbatim,enumerate,tikz,url,hyperref,mathrsfs,bbm,cite}
\usepackage[top=1in,bottom=1in,left=1in,right=1in,includefoot,footskip=30pt]{geometry}

\usepackage[english,algoruled,vlined,resetcount,linesnumbered]{algorithm2e}

\SetAlgoProcName{Algorithm}{Algorithm}
\SetCommentSty{textit}
\SetProcNameSty{textsc}
\SetKwFor{ForEach}{for each}{do}{end}

\usetikzlibrary{shapes.misc,calc,intersections,patterns,decorations.pathreplacing}
\hypersetup{colorlinks=true,citecolor=blue,linktoc=page}

\usepackage[abbrev,msc-links]{amsrefs}
\usepackage{doi}

\renewcommand{\PrintDOI}[1]{\doi{#1}}

\AtBeginDocument{\def\MR#1{}}

\usepackage{tikz}
\usetikzlibrary{decorations.markings}
\usetikzlibrary{calc,positioning,decorations.pathmorphing,decorations.pathreplacing}

\usepackage{fullpage}
\usepackage{setspace}
\usepackage{datetime}
\usepackage[margin=1cm]{caption}

\newtheorem{theorem}{Theorem}
\newtheorem{lemma}[theorem]{Lemma}
\newtheorem{prop}[theorem]{Proposition}
\newtheorem{claim}[theorem]{Claim}

\newtheorem{conj}[theorem]{Conjecture}

\theoremstyle{definition}
\newtheorem{definition}[theorem]{Definition}

\usepackage{mathtools}

\DeclarePairedDelimiterXPP{\prob}[1]{\mathbb{P}}{(}
{)}{}{#1}
\DeclarePairedDelimiterXPP{\expec}[1]{\mathbb{E}}{\lbrack}
{\rbrack}{}{#1}

\let\polishlcross=\l
\def\l{\ifmmode\ell\else\polishlcross\fi}

\makeatletter
\def\moverlay{\mathpalette\mov@rlay}
\def\mov@rlay#1#2{\leavevmode\vtop{    \baselineskip\z@skip\lineskiplimit-\maxdimen%
    \ialign{\hfil$\m@th#1##$\hfil\cr#2\crcr}}}
\newcommand{\charfusion}[3][\mathord]{
  #1{\ifx#1\mathop\vphantom{#2}\fi
    \mathpalette\mov@rlay{#2\cr#3}
  }
  \ifx#1\mathop\expandafter\displaylimits\fi}
\makeatother

\DeclareFontFamily{U}  {MnSymbolC}{}
\DeclareSymbolFont{MnSyC}         {U}  {MnSymbolC}{m}{n}
\DeclareFontShape{U}{MnSymbolC}{m}{n}{%
  <-6>  MnSymbolC5
  <6-7>  MnSymbolC6
  <7-8>  MnSymbolC7
  <8-9>  MnSymbolC8
  <9-10> MnSymbolC9
  <10-12> MnSymbolC10
  <12->   MnSymbolC12}{}
\DeclareMathSymbol{\powerset}{\mathord}{MnSyC}{180}

\makeatletter
\def\namedlabel#1#2{\begingroup
  #2%
  \def\@currentlabel{#2}%
  \phantomsection\label{#1}\endgroup
}
\makeatother

\numberwithin{theorem}{section}
\everydisplay{\thickmuskip=7mu}
\renewcommand{\leq}{\leqslant}
\renewcommand{\geq}{\geqslant}

\renewcommand{\to}{\rightarrow}

\usepackage{accents,graphicx,bm}
\newcommand{\dvec}[1]{\accentset{\scalebox{.4}{\bm{$\overleftrightarrow\null$}}}{#1}}
\newcommand\dlr{\dvec{d}}
\newcommand\elr{\dvec{e}}

\newcommand{\ex}{\operatorname{ex}}

\let\epsilon\varepsilon%

\def\cC{\mathcal{C}}
\def\cD{\mathcal{D}}

\def\cS{\mathcal{S}}

\def\Ex{\mathbb{E}}

\def\N{\mathbb{N}}

\def\1{\mathbbm{1}}
\def\<{\langle}
\def\>{\rangle}

\def\circ{C^\circlearrowright}
\def\ex{\mathop{\text{\rm ex}}\nolimits}

\let\theta=\vartheta%
\let\rho=\varrho%
\let\phi=\varphi%

\usepackage{cases}
\usepackage{eqparbox}

\begin{document}
\onehalfspace%
\shortdate%
\yyyymmdddate%
\settimeformat{ampmtime}
\footskip=28pt

\title{Counting orientations of random graphs\\ with no directed $k$-cycles}

\author[M.~Campos]{Marcelo Campos}

\author[M.~Collares]{Maur\'{\i}cio Collares}

\author[G.~O.~Mota]{Guilherme Oliveira Mota}

\address{Trinity College, Cambridge CB2 1TQ, United Kingdom}
\email{mc2482@cam.ac.uk}

\address{Institute of Discrete Mathematics, Graz University of Technology, Steyrergasse 30, 8010 Graz, Austria}
\email{mauricio@collares.org}

\address{Instituto de Matem\'atica e Estat\'{\i}stica, Universidade de
  S\~ao Paulo, S\~ao Paulo, Brazil}
\email{mota@ime.usp.br}

\thanks{M.
  Campos was supported by CNPq and FAPERJ\@.
  M. Collares was supported by CNPq (406248/2021-4).
  This research was funded in part by the Austrian Science Fund (FWF) P36131.
  For the purpose of open access, the author has applied a CC BY public copyright licence to any Author Accepted Manuscript version arising from this submission.
  G. O. Mota was supported by CNPq (306620/2020-0, 406248/2021-4) and FAPESP (2018/04876-1, 2019/13364-7).
  This study was financed in part by the Coordena\c{c}\~ao de Aperfei\c{c}oamento de Pessoal de N\'ivel Superior, Brazil (CAPES), Finance Code 001.
}

\begin{abstract}
  For every $k\geq 3$, we determine the order of growth, up to polylogarithmic factors, of the number of orientations of the binomial random graph containing no directed cycle of length~$k$.
  This solves a conjecture of Kohayakawa, Morris and the last two authors.
\end{abstract}

\maketitle
\vspace{-0.4cm}
\section{Introduction}

An orientation $\vec{H}$ of a graph $H$ is an oriented graph obtained
by assigning an orientation to each edge of $H$.  The study of the
number of $\vec{H}$-free orientations of a graph~$G$, denoted by
$D(G,\vec{H})$, was initiated by Erd\H{o}s~\cite{Er74}, who posed the
problem of determining $D(n,\vec H) := \max\big\{D(G,\vec H) : |V(G)|
=n \big\}$.  For tournaments, this problem was solved by Alon and
Yuster~\cite{AlYu06}, who proved that~$D{(n,\vec{T}) = 2^{\ex(n,K_k)}}$
holds for any tournament~$\vec{T}$ on $k$ vertices if $n \in \N$ is sufficiently large as a function of $k$.

Let $\circ_k$ denote the directed cycle of length $k$.
Buci\'c, Janzer and Sudakov~\cite{BuJaSu23} determined $D(n,\circ_{2\ell+1})$ for every $\ell\geq 1$ as long as $n$ is sufficiently large, extending the proof in~\cite{AlYu06}.
Another extension of the results in~\cite{AlYu06} was given by Ara\'ujo, Botler, and the last author~\cite{ArBoMo20+} who determined~$D(n,\circ_3)$ for every $n \in \mathbb{N}$ (see also~\cite{BoHoMo22}).

In the context of random graphs, Allen, Kohayakawa, Parente, and the last author~\cite{AlKoMoPa14} investigated the problem of determining the typical number of $\circ_{k}$-free orientations of the Erd\H{o}s--R\'enyi random graph~$G(n,p)$.
They proved that, for every $k\geq 3$, with high probability as $n\to\infty$ we have $\log_2 D(G(n,p), \circ_{k}) = o(pn^2)$ for $p \gg n^{-1 + 1/(k-1)}$, and $\log_2 D(G(n,p), \circ_{k}) = (1+o(1))p\binom{n}{2}$ for $n^{-2}\ll p \ll n^{-1 + 1/(k-1)}$.
This result was improved in the case of triangles by Kohayakawa, Morris and the last two authors~\cite{CoKoMoMo20}, who proved, among other things, the following result\footnote{The $\widetilde{\Theta}(\cdot)$ and $\widetilde{O}(\cdot)$ notation are analogous to $\Theta$ and $O$ notation but with polylogarithmic factors omitted.
For convenience, from now on $\log$ will denote the natural logarithm.}.

\begin{theorem}[\cite{CoKoMoMo20}*{Theorem~1.2}] \label{thm:triangles}
  If $p \gg n^{-1/2}$, then, with high probability as $n \to \infty$,
  \[
    \log D(G(n,p), \circ_{3}) = \widetilde{\Theta}\big(n / p \big)\nonumber.
  \]
\end{theorem}

For general $k$, one can show\footnote{Recall that Harary and Moser~\cite[Theorem 7]{HararyMoser1966} observed that every strongly connected tournament with $t$ vertices contains a cycle of length $i$ for every $i = 3, \ldots, t$, which implies that every strongly connected component of a $\circ_k$-free tournament has size at most $k-1$.
 Therefore, to count $\circ_k$-free tournaments, it suffices to consider all ordered partitions $V_1 \cup \cdots \cup V_p$ of $[n]$ with parts of size less than $k$ and count tournaments whose strongly connected components respect this partition and its order.
 There are at most $2^n \cdot n!$ ordered partitions of $[n]$ and at most $(k-2)n/2$ edges inside parts to orient.} that $D(K_n, \circ_k) \leq 2^{kn} \cdot n!$ for every $k \geq 3$. A first step towards determining~$\log D(G(n,p), \circ_{k})$ for $k\geq 4$ was also given in~\cite{CoKoMoMo20}, where it was proved that
\begin{equation}\label{eq:weak_upper_bound}
  \log D(G(n,p), \circ_{k}) = \widetilde{O}\big(n / p \big)
\end{equation}
with high probability.
Moreover, they proved that a natural generalisation of the lower bound construction used in the proof of Theorem~\ref{thm:triangles} gives
\begin{equation}\label{lemma:lower_bound}
  \log D(G(n, p), \circ_{k}) = \Omega\left(\frac{n}{p^{1/(k-2)}}\right)
\end{equation}
with high probability when $p \gg n^{-1+1/(k-1)}$.
They conjectured that \eqref{lemma:lower_bound} is sharp up to polylogarithmic factors, and we confirm this conjecture by proving the following result.

\begin{theorem}\label{thm:cycles}
  Let $k\geq 3$ and $p = p(n) \gg n^{-1+1/(k-1)}$.
  Then,  with high probability as $n \to \infty$,
  \[
    \log D\big(G(n,p), \circ_{k} \big) = \widetilde{\Theta}\bigg( \frac{n}{p^{1/(k-2)}} \bigg).
  \]
\end{theorem}
The proof of Theorem~\ref{thm:cycles} will be outlined in the next section, but we present here a short overview of three key ideas in the proof.
We will be interested in an auxiliary graph encoding $(k-2)$-paths, since we will be able to encode the orientation of the neighbourhood of each vertex as an independent set in this graph. One of the main challenges will be ensuring this auxiliary graph is dense enough so that we may apply the graph container lemma.
To do so, the first key idea is to define, for each $1
\leq r \leq k-2$, a pseudorandomness condition on the number of
directed $r$-paths between small sets, and split the proof into $k-2$
cases according to whether such a pseudorandomness condition holds for a given
$r$.  We design these conditions in such a way that the case $r=1$ holds for any orientation of $G(n,p)$ with high probability, and such that we may finish the proof using the graph container method (introduced by Kleitman and Winston~\cite{KlWi82}, and rediscovered and developed by Sapozhenko~\cite{Sa01}) if the condition holds for $r = k-2$. The second key idea is how to deal with orientations that do not satisfy the pseudorandomness condition for some $2 \leq r \leq k-2$.
In this situation, we provide a way to efficiently ``encode'' the orientations that do not satisfy the condition for some value of $r$ but do for all smaller values.

We remark that, to implement the above two ideas, we need to construct the orientation ``online'', that is, to consider a subgraph $H \subset G(n,p)$, and for each possible orientation $\vec{H}$ of $H$, reveal the randomness between a new subset of vertices and $V(H)$ and consider all ways to extend $\vec{H}$.
Counting orientations of $G(n,p)$ is, however, an ``offline'' problem: when exposing $G(n,p)$ in multiple rounds, the orientation of an early-round edge may depend on the randomness of later rounds.
To circumvent this fact, we use our third key idea, which is to count the expected number of orientations of $G(n,p)$ that are $\circ_{k}$-free.
When estimating this expectation we will be able to split it in a way that makes possible to proceed inductively and use the randomness in steps after orienting part of the edges.

\section{Outline of the proof of Theorem~\ref{thm:cycles}}\label{sec:outline}

We start by presenting a short sketch of the proof of
Theorem~\ref{thm:cycles} for $\circ_3$, as proved
in~\cite{CoKoMoMo20}. This will motivate the idea behind the proof for
general directed cycles $\circ_k$.

\subsection{The proof in~\cite{CoKoMoMo20}}\label{subset:CoKoMoMo}
The idea is to obtain, by induction on the number of vertices, the
following upper bound on the number of $\circ_3$-free orientations of an
$n$-vertex graph~$G$:
\begin{equation}
  \label{eq:alphaG}
\binom{n}{\leq \alpha(G)}^{2n},
\end{equation}
where $\alpha(G)$ denotes the independence number of $G$, and $\binom{n}{\leq t}$ is shorthand for $\sum_{i=0}^t \binom{n}{i}$.
In order to obtain such a bound, let $G$ be a graph on vertex set $V$ and consider $v\in V$.
Let $H=G \setminus \{v\}$ and suppose that the number of $\circ_3$-free orientations of $H$ is
\begin{equation}
  \label{eq:alphaH}
\binom{n-1}{\leq \alpha(H)}^{2(n-1)}.
\end{equation}
Then, for each $\circ_3$-free orientation $\vec{G}$ of $G$, let $\vec{H}$ be its restriction to $H$ and pick $\vec{T}\subset E(\vec{G})\setminus E(\vec{H})$ minimal such that $\vec{G}$ is the only $\circ_3$-free orientation of $G$ containing $\vec{T}\cup \vec{H}$.
The key observation is that, by the minimality of $T$, the vertex sets $N^+_{\vec{T}}(v)$ and $N^-_{\vec{T}}(v)$ are independent sets in $H$, so there are at most $\binom{n}{\leq \alpha(G)}^2$ choices for $\vec{T}$.
This together with~\eqref{eq:alphaH} and $\alpha(H)\leq \alpha(G)$ completes the proof of~\eqref{eq:alphaG}.
One may check that the expected number of independent sets of size $(3\log n)/p$ in $G(n, p)$ tends to zero with $n$, and therefore \eqref{eq:alphaG} and the first moment method imply Theorem~\ref{thm:cycles} for $\circ_3$.

We will generalise the ideas depicted above in two ways, which will be
described in the next two subsections.  Directed paths of length
${k-2}$ starting or ending in the neighbourhood of a vertex will play
a key role in our proofs.  For this reason, we fix $\ell = k-2 \geq
1$, and our aim is to count orientations of $G(n,p)$ avoiding copies
of $\circ_{\ell+2}$.

\subsection{Pseudorandomness}
To count the desired orientations of $G(n,p)$, we define a
``pseudorandom'' oriented graph property, and we proceed to separately
count the $\circ_{\ell+2}$-free orientations depending on whether some
already-oriented subgraph $\vec{H}$ is pseudorandom.  We then use the
randomness of $G(n,p)$ in a stronger way, exposing the randomness bit
by bit in each step of the induction.

Let us discuss the pseudorandom property we mentioned in the previous paragraph.
We write $\vec{P}_r$ for the directed path with $r$ edges and, given an oriented graph $\vec{G}$, we denote by $\vec{G}^r$ the digraph such that $(u, v)$ is an edge whenever there is a $\vec{P}_r$ from $u$ to $v$ in $\vec{G}$.
Roughly speaking, we say an oriented graph $\vec{G}$ is \emph{$r$-locally dense} (for a complete description of this property, see Definition~\ref{def:locally_dense}), if
\begin{equation}\label{eq:roughLocally}
  \elr_{(\vec{G} \setminus X)^r}(A',B)\geq \frac{1}{2} \cdot p^{\frac{\ell-r+1}{\ell}}|A'||B|
\end{equation}
for all disjoint sets $A',B,X\subset V(G)$ of size roughly $(\log n)/p$, where $\vec{G} \setminus X$ is a shorthand for $\vec{G}[V(G) \setminus X]$ and $\elr_{(\vec{G} \setminus X)^r}(A',B)$ denotes the number of edges between $A'$ and $B$ (in either direction) in the digraph $(\vec{G} \setminus X)^r$.

Observe that being $1$-locally dense does not depend on the orientation of the graph (and the set $X$ plays no role in this case), i.e., being $1$-locally dense is a pseudorandom property that depends only on the underlying graph $G$.
By Chernoff's inequality, any orientation of $G(n,p)$ is $1$-locally dense with high probability, and one may think of this property as a strengthening of the property that $\alpha(G(n,p))\leq (3\log n)/p$.

\subsection{Sketch of the proof}

We will count separately the orientations which are
$\ell$-locally dense and the orientations which are not $r$-locally
dense but are $(r-1)$-locally dense for some $2 \leq r \leq \ell$.
To count the $\circ_{\ell+2}$-free, $\ell$-locally dense orientations $\vec{G}$ (see
Lemma~\ref{lemma:key}~\eqref{lemma:locally-dense}), we proceed
similarly to the proof in~\cite{CoKoMoMo20}: let $v\in V$, put $H=G
\setminus\{v\}$, and let $\vec{H}$ be the restriction of $\vec{G}$ to $H$.
Moreover, let $\vec{T}\subset E(\vec{G})\setminus E(\vec{H})$
be minimal such that $\vec{G}$ is the only $\circ_{\ell+2}$-free
orientation of $G$ containing $\vec{T}\cup \vec{H}$.  Note that, since
we are avoiding copies of $\circ_{\ell+2}$, the sets
$T^+:=N^+_{\vec{T}}(v)$ and $T^-:=N^-_{\vec{T}}(v)$ are independent
sets in $\vec{H}^\ell$ by minimality of $\vec{T}$.  Since $\vec{G}$ is
$\ell$-locally dense, the edge density of $\vec{H}^\ell$ is at least
of order $p^{1/\ell}$.  This allows us to prove that the largest
independent set in $N_G(v)$ has size roughly $(\log n)/ p^{1/\ell}$,
which will be enough to finish the proof of this case using the graph
container lemma.

To count orientations $\vec{G}$ that are not $r$-locally dense but are
$(r-1)$-locally dense (see
Lemma~\ref{lemma:key}~\eqref{lemma:locally-sparse}), we use the fact
that there are ``large'' disjoint sets $A',B\subset V$ and $A\subset A'$,
with $|A|\geq |A'|/2$, such that for every $a\in A$ it holds that
$\dlr_{\vec{G}^r}(a,B) < p^{\frac{\ell-r+1}{\ell}}|B|$, where
$\dlr(a,B)$ denotes $d^+(a,B) + d^-(a,B)$ (for simplicity, in this
outline we assume the set $X$ in \eqref{eq:roughLocally} is empty).
Put $\vec{H}=\vec{G} \setminus A$ and note that, since $\vec{G}$ is
$(r-1)$-locally dense, $\vec{H}^{r-1}$ has ``many'' edges between any
two ``sufficiently large'' sets\footnote{Note that we want $\vec{H}^{r-1}=(\vec{G} \setminus A)^{r-1}$ to have many
  edges between pairs of sets.
This does not directly follow from $\vec{G}$ being $(r-1)$-locally-dense, because many $(r-1)$-paths could pass through $A$.
The set $X$ in the definition of $r$-locally-dense graphs is used to handle this issue.}.

Now given $a\in A$ we may choose $S^+\subset V(H)\cap N_{\vec{G}}^+(a)$, the set of all $v\in N_{\vec{G}}^+(a)$ such that \[ d^+_{\vec{H}^{r-1}}(v,B)\geq \frac{1}{4} \cdot p^{\frac{\ell-r+2}{\ell}}|B|\] and $S^-\subset V(H)\cap N_{\vec{G}}^-(a)$, the set of all $v\in N_{\vec{G}}^-(a)$ such that \[d^-_{\vec{H}^{r-1}}(v,B) >\frac{1}{4} \cdot p^{\frac{\ell-r+2}{\ell}}|B|.\] We claim that $|S^+| + |S^-|\leq 4p^{-1/\ell}$.
Indeed, every $v \in S^+ \cup S^-$ corresponds to at least $p^{\frac{\ell-r+2}{\ell}}|B|/4$ paths of length $r$ starting with the edge $av$ or ending with the edge $va$, so if $|S^+| + |S^-| > 4p^{-1/\ell}$ then $\dlr_{\vec{G}^r}(a,B)$ would be too large.
Moreover, we are able to use the randomness of $G$ to show that $S^+$ and $S^-$ determine the orientation of all but a very small number of edges of $G$ between~$a$ and~$V(H)$.

\section{The main result}\label{sec:forb-direct-cycl}

In this section our goal is to count orientations of $G(n,p)$ containing no copies of $\circ_{\ell+2}$.
We prove the following result, which generalises the upper bound of Theorem~\ref{thm:cycles} to the case where $\ell$ is a function of $n$ and provides an explicit bound on the number of orientations.

\begin{theorem}\label{thm:main}
 Let $\ell = \ell(n)\in \N$ and $p = p(n)$ be such that $0<p\leq (2^8 \ell)^{-\ell}$. With high probability as $n \to \infty$, $G(n, p)$ admits at most \[\exp\left(\frac{13\ell n(\log n)^2}{p^{1/\ell}}\right)\] $\circ_{\ell+2}$-free orientations.
\end{theorem}

We remark that, since~\eqref{eq:weak_upper_bound} implies the upper bound in Theorem~\ref{thm:cycles} when $p$ and $\ell$ are constant, proving Theorem~\ref{thm:main} indeed suffices to complete the proof of Theorem~\ref{thm:cycles} despite the extra condition $p \leq (2^8 \ell)^{-\ell}$.
The upper bound on $p$ will be needed in the proof of Lemma~\ref{lemma:key}(\ref{lemma:locally-sparse}), more specifically in Claim~\ref{claim:maxdeg_T}.
We present the proof of Theorem~\ref{thm:main} (assuming the validity of Lemma~\ref{lemma:key}, which will be proved in Sections~\ref{sec:locally-sparse} and~\ref{sec:locally-dense}) at the end of the section.

For any oriented graph $\vec{G}$, its underlying undirected graph will be denoted by $G$.
We recall other useful notation introduced in Section~\ref{sec:outline}: Given a graph $\vec{G}$, the digraph $\vec{G}^r$ contains the edge $(u, v)$ precisely whenever there is an oriented path of length $r$ starting at $u$ and ending at $v$ in $\vec{G}$.
Moreover, $\vec{G} \setminus X$ is a shorthand for $\vec{G}[V(G) \setminus X]$, and, for disjoint sets $A, B \subset V(G)$, $\elr_{\vec{G}}(A,B)$ denotes the number of edges between $A$ and $B$ (in either direction) in $\vec{G}$.

The following definition will be used to ``encode'' orientations of $\circ_{\ell+2}$-free graphs.
It implicitly depends on a parameter $\alpha$ which will be chosen later.

\begin{definition}\label{def:locally_dense}
  Given $1 \leq r\leq \ell$, an oriented graph $\vec{G}$ is \emph{$r$-locally dense} if for every pairwise disjoint sets $A'$, $B$, $X$ of $V(G)$ such that
\begin{equation}\label{eq:T_r}
|A'| = \alpha,\qquad r\alpha\leq |B|\leq \ell\alpha,\qquad \text{and} \qquad |X|\leq (\ell+1-r) \alpha,
\end{equation}
we have
\begin{equation}\label{eq:locally_dense}
  \elr_{(\vec{G} \setminus X)^r}(A',B)\geq \frac{1}{2} \cdot p^{\frac{\ell-r+1}{\ell}}|A'||B|.
\end{equation}
Otherwise, the orientation $\vec{G}$ is called \emph{$r$-locally-sparse}.
\end{definition}

Even though the following lemma is a trivial application of Chernoff's inequality together with the fact that the definition of $1$-locally-dense depends solely on the underlying undirected graph and not on the orientation of the edges, it will be crucial.

\begin{lemma}\label{lemma:1-locally-dense}
  With high probability every orientation of $G(n, p)$ is $1$-locally-dense for $\alpha = 2^6 (\log n)/p$.
\end{lemma}

Given a graph $G$ and an induced subgraph $H$ of $G$, an orientation $\vec{G}$ of $G$ is an \emph{extension} of an orientation $\vec{H}$ of $H$ if $E(\vec{H}) \subset E(\vec{G})$.
Furthermore, we say that $\vec{G}$ \emph{extends} $\vec{H}$.
Due to Lemma~\ref{lemma:1-locally-dense}, in the rest of the paper we count $1$-locally-dense $\circ_{\ell+2}$-free orientations.

\begin{definition}\label{def:D}
  Let $G=(V,E)$ be a graph and let $\vec{H}$ be an orientation of an induced subgraph $H$ of $G$.
  We denote by $\cD_r(G,\vec{H})$ (resp.\ $\cS_r(G, \vec{H})$) the set of all $1$-locally-dense, $\circ_{\ell+2}$-free orientations of $G$ that extend $\vec{H}$ and are $r$-locally-dense (resp.\ $r$-locally-sparse).
  For convenience, we also write $\cD_r(G)$ for $\cD_r(G, \emptyset)$ and $\cS_r(G)$ for $\cS_r(G, \emptyset)$.
\end{definition}

In view of Lemma~\ref{lemma:1-locally-dense} and using the language described in Definition~\ref{def:D}, our goal is to prove that $|\cD_1(G(n,p))| \leq \exp\big(13\ell n(\log n)^2/p^{1/\ell}\big)$ holds with high probability.
The rest of the paper will be dedicated to this task, and from now on $\ell$, $n$ and $p$ will be fixed, all graphs will have vertex set contained in $[n]$, and we set
\[
\alpha := 2^6(\log n)/p.
\]

Note that if $\vec{G} \in \cS_r(G)$, then there exist pairwise
disjoint sets $A'$, $B$ and $X$ of $V(G)$ satisfying~\eqref{eq:T_r}
such that~\eqref{eq:locally_dense} fails to hold.  This fact implies
the existence of $A \subset A'$ with $|A| = |A'|/2$ such that
$\dlr_{(\vec{G} \setminus X)^r}(a,B) \leq
p^{\frac{\ell-r+1}{\ell}}|B|$ for every $a \in A$.  This motivates
the following definition, where $H$ will play the role of $G\setminus
A$.

 \begin{definition}\label{def:witness}
  Let $H$ be a graph, $A\subset [n]\setminus V(H)$ and let $B$ and $X$
  be disjoint subsets of $V(H)$.  The quadruple $(A, H, B, X)$ is an
   \emph{$r$-frame} if
   \[
     |A| = \alpha/2,\qquad r\alpha\leq |B|\leq \ell\alpha,\qquad \text{and} \qquad |X|\leq (\ell+1-r) \alpha.
   \]
  Given an $r$-frame $(A, H, B, X)$, an orientation $\vec{H}$ of $H$, and a graph $G$ with $A=V(G)\setminus V(H)$,
  an orientation $\vec{G}$ of $G$ is said to be an \emph{$(r, B, X)$-sparse
    extension} of $\vec{H}$ if it holds that $\dlr_{(\vec{G} \setminus X)^r}(a,B)
  \leq p^{\frac{\ell-r+1}{\ell}}|B|$ for every $a\in A$. 
\end{definition}

We let $\cS_r(G, \vec{H}, B, X)$ denote the set of all orientations of $\vec{G}$ that are $(r,B,X)$-sparse extensions of $\vec{H}$.
Observe that we have
\begin{equation}\label{eq:dense_sparse_decomposition}
  \cS_r(G) =  \bigcup_{(A, B, X)} \bigcup_{\vec{H} \in \cD_{1}(G\setminus A)} \cS_r(G, \vec{H}, B, X),
\end{equation}
where the first union is over all $(A, B, X)$ such that $(A, G \setminus A, B, X)$ is an $r$-frame.

Given a graph $H$ and a set $A$ disjoint from $V(H)$, we define $G(A,H,p)$ as the random graph $G$ on vertex set $V(G) = A \cup V(H)$ such that $G[V(H)]=H$ and such that each element of $\binom{V(G)}{2}\setminus \binom{V(H)}{2}$ is present in $E(G)$ with probability $p$ independently at random.
We are now ready to state the main tool for the proof of Theorem~\ref{thm:main}, consisting of two bounds on the expected number of extensions of a digraph $\vec{H}$.
Its proof is postponed to Sections~\ref{sec:locally-sparse} and~\ref{sec:locally-dense}.

\begin{lemma}\label{lemma:key}
  Let $\vec{H}$ be an oriented graph and $A \subset [n] \setminus V(H)$ of size $\alpha/2$. The following holds for $G=G(A,H,p)$.
  \begin{enumerate}[$(i)$]
  \item \label{lemma:locally-sparse}
    If $2\leq r\leq \ell$ and $B, X \subset V(H)$ are such that $(A,H,B,X)$ is an $r$-frame, then
     \[\Ex \left[\big|\cD_{r-1}(G) \cap
       \cS_r(G,\vec{H},B,X)\big|\right]\leq \exp\left(9\alpha
       p^{-\frac{1}{\ell}}\log n\right). \]
   \item \label{lemma:locally-dense}
     $\Ex\left[\big|\cD_\ell(G, \vec{H})\big|\right]\leq \exp\left(\alpha (\ell+2) p^{-\frac{1}{\ell}}(\log n)^2\right).$
\end{enumerate}
\end{lemma}

Before proving Theorem~\ref{thm:main}, we state a simple probabilistic estimate which will
be used a few times throughout the paper.

\begin{lemma}\label{lemma:chernoff}
  Let $0 \leq p \leq 1$ and $C > 0$. If $S_p$ is a $p$-random subset of a finite set $S$, then
  \[
    \Ex\left[ C^{|S_p|}\right]\leq e^{Cp|S|}.
  \]
\end{lemma}

\begin{proof}
  Since $\mathbb{P}(|S_p| = t) \leq \big(p|S|\big)^t/t!$ for every $t \geq 0$, we can compute
  \[
    \Ex[C^{|S_p|}] \leq \sum_{t=0}^{\infty} \frac{\big(Cp|S|\big)^t}{t!} = e^{Cp|S|},
  \]
  as claimed.
\end{proof}

We are ready to prove Theorem~\ref{thm:main}.

\begin{proof}[Proof of Theorem~\ref{thm:main}]
  Recall that $n$ is fixed, and put $z:=\exp\big(4(\ell+2) p^{-\frac{1}{\ell}}(\log n)^2\big)$.
  We will show by induction on $n' \leq n$ that
\begin{equation}\label{eq:main_induction}
    \Ex \Big[\big|\cD_1(G(n',p))\big|\Big] \leq z^{n'},
\end{equation}
which together with Lemma~\ref{lemma:1-locally-dense} and Markov's
inequality implies the desired result since $4(\ell+2) \leq 12\ell$.
Note that if $n' \leq p^{-1-1/\ell} (\log n)^2$ holds, then an application of Lemma~\ref{lemma:chernoff} gives
\[
  \Ex \Big[\big|\cD_1(G(n', p))\big|\Big]\leq \Ex\Big[2^{e(G(n',p))}\Big]\leq \exp\Big(p(n')^2\Big)\leq z^{n'}.
\]
Therefore, we may assume $n' > p^{-1-1/\ell}(\log n)^2$, which implies $n' \gg \alpha$.
We claim that it is enough to prove that
\begin{equation}
  \label{eq:expec}
  \Ex\Big[\big|\mathcal{D}_1(G(n',p))\big|\Big]\leq  z^{\alpha/2}\cdot \Ex \Big[\big|\mathcal{D}_1(G(n'-\alpha/2,p))\big|\Big].
\end{equation}
Indeed, by~\eqref{eq:expec} and the induction hypothesis
we obtain
\[
    \Ex \Big[\big|\cD_1(G(n',p))\big|\Big] \leq z^{\alpha/2}\cdot \Ex \Big[\big|\cD_1(G(n'-\alpha/2,p))\big|\Big] \leq z^{\alpha/2} \cdot z^{n'-\alpha/2}= z^{n'},
\]
which proves~\eqref{eq:main_induction}.
Thus, the remainder of the proof is dedicated to showing that~\eqref{eq:expec} holds.
Let $G=G(n',p)$ and notice that, since every $1$-locally-dense
orientation is either $\ell$-locally dense or admits a minimal $2 \leq
r \leq \ell$ for which it is $r$-locally sparse, we have
  \begin{equation}\label{eq:1_dense_decomp}
    \cD_1(G) = \cD_\ell(G) \; \cup \; \bigcup_{r=2}^{\ell} \,\, (\cD_{r-1}(G) \cap \cS_r(G)).
  \end{equation}
  Thus, it suffices to upper bound the expected sizes of the sets in the right-hand side of~\eqref{eq:1_dense_decomp}.
  Let $A\subset V(G)$ be an arbitrary set of size $\alpha/2$, and let $H = G\setminus A$ and $F$ be the graph with $V(F)=V(G)$ and  $E(F)=E(G)	 \setminus \binom{V(H)}{2}$.
  Observe that $F$ and $H$ independent of each other as random graphs.
  Therefore,
  \begin{equation}\label{eq:dense_bound}
    \Ex\Big[\big| \cD_\ell(G) \big|\Big] = \Ex_H\Bigg[\Ex_F\Big[ \sum_{\vec{H}} \big| \cD_\ell(G, \vec{H}) \big| \;\Big|\; H\Big]\Bigg] =  \Ex_H\Bigg[\sum_{\vec{H}}\Ex_F\Big[  \big| \cD_\ell(G, \vec{H}) \big| \;\Big|\; H\Big]\Bigg],
  \end{equation}
  where the sums are over $\vec{H} \in \cD_1(H)$.
  Conditioned on $H$ (that is, on $G \setminus A = H$), the distribution of $F \cup H$ is the same as $G(A, H, p)$.
  Then, from Lemma~\ref{lemma:key}\eqref{lemma:locally-dense} we obtain
  \[
    \Ex_F\Big[  \big| \cD_\ell(G, \vec{H}) \big| \;\Big|\; H\Big] \leq z^{\alpha/4},
  \]
  and hence, by~\eqref{eq:dense_bound},
  \[
    \Ex\Big[\big| \cD_\ell(G) \big|\Big] \leq \Ex_H\Bigg[\sum_{\vec{H} \in \cD_1(H)} z^{\alpha/4}\Bigg] = z^{\alpha/4} \cdot \Ex\Big[\big|\cD_1(G(n' - \alpha/2), p)\big|\Big].
  \]
  To bound the expected sizes of the remaining sets of the right-hand side of~\eqref{eq:1_dense_decomp}, we proceed analogously. For any $2\leq r\leq \ell$, using~\eqref{eq:dense_sparse_decomposition} we have
  \begin{align*}
    \Ex\Big[\big|\cD_{r-1}(G) \cap \cS_r(G)\big|\Big] &\leq \;\Ex\; \Big[ \sum_{(A,B,X)} \sum_{\vec{H} \in \cD_1(H)} \big|\cD_{r-1}(G) \cap \cS_r(G, \vec{H}, B, X)\big| \Big] \\
                                                      &= \sum_{(A,B,X)} \Ex_H\Bigg[\sum_{\vec{H} \in \cD_1(H)} \Ex_F \Big[ \big|\cD_{r-1}(G) \cap \cS_r(G, \vec{H}, B, X)\big| \;\Big|\; H\Big]\Bigg],
  \end{align*}
  where the outer sums are over all triples $(A, B,X)$ such that $(A, H, B, X)$ is an $r$-frame for $H = G \setminus A$. Since there are at most $n^{(2\ell+1)\alpha} \leq z^{\alpha/4}$ such triples, an application of
  Lemma~\ref{lemma:key}\eqref{lemma:locally-sparse} gives
  \[
    \Ex\Big[\big|  \cD_{r-1}(G) \cap \cS_r(G) \big|\Big] \leq \sum_{(A,B,X)} \Ex_H\Bigg[\sum_{\vec{H} \in \cD_1(H)} z^{\alpha/12}\Bigg] \leq z^{\alpha/3} \cdot \Ex\Big[\big|\cD_1(G(n' - \alpha/2), p)\big|\Big].
  \]
  By \eqref{eq:1_dense_decomp}, we thus have $$ \Ex\Big[\big|\mathcal{D}_1(G(n',p))\big|\Big]\leq \ell z^{\alpha/3}\cdot \Ex \Big[\big|\mathcal{D}_1(G(n'-\alpha/2,p))\big|\Big] $$
  which verifies~\eqref{eq:expec} and concludes the proof of~Theorem~\ref{thm:main}, since $\ell\leq p^{-1/\ell} \leq z^{\alpha/6}$.
\end{proof}

\section{Extending locally-sparse orientations}\label{sec:locally-sparse}

The main goal of this section is to prove
Lemma~\ref{lemma:key}(\ref{lemma:locally-sparse}), which bounds the
expected number of $(r-1)$-locally dense, $(r, B, X)$-sparse
extensions of an oriented graph $\vec{H}$.
The deterministic part of the proof is the following proposition.
We remark that, although we intentionally state the result in a way that emphasises its container-type nature, its proof does not use the Hypergraph Container Lemma.

\begin{prop}\label{prop:container-like}
  Let $\vec{H}$ be an oriented graph and $2 \leq r \leq
  \ell$. If $(A,H,B,X)$ is an $r$-frame, then there exists a family $\mathcal{C}$ of digraphs on
  vertex set $A \cup V(H)$ such that
  \begin{enumerate}[$(i)$]
  \item \label{item:contains} For every graph $G$ with $V(G) = A \cup
    V(H)$ and every orientation $\vec{G} \in \mathcal{D}_{r-1}(G) \cap
    \mathcal{S}_r(G, \vec{H}, B, X)$, there exists some $\vec{C} \in
    \mathcal{C}$ such that $E_{\vec{G}}(A, V(H)) \subset E(\vec{C})$.
  \item \label{item:few-antiparallel} For every $\vec{C} \in \mathcal{C}$, at most $\ell\alpha^2$ pairs $\{u,v\} \subset V(\vec{C})$ satisfy $\{uv, vu\} \subset E(\vec{C})$,
  \item \label{item:few-containers} $|\mathcal{C}| \leq \exp(8\alpha p^{-1/\ell} \log n)$.
  \end{enumerate}
\end{prop}
In other words, we find a small family of digraphs (by property (\ref{item:few-containers})) which cover every orientation of $G$ in $\mathcal{D}_{r-1}(G) \cap \mathcal{S}_r(G, \vec{H}, B, X)$ (by property (\ref{item:contains})). Additionally, for each $\vec{C} \in \mathcal{C}$, any two orientations of $G$ covered by $\vec{C}$ differ only in a small set of edges (by property (\ref{item:few-antiparallel})), allowing us to bound the expected number of orientations of $G = G(n,p)$ contained in each $\vec{C}$. We will then be able to bound the expected size of $\mathcal{D}_{r-1}(G) \cap \mathcal{S}_r(G, \vec{H}, B, X)$ by summing the expected number of orientations of $G$ contained in $\vec{C}$ for each $\vec{C} \in \mathcal{C}$.

\begin{proof}[Proof of Proposition~\ref{prop:container-like}]
  The proof loosely follows the idea behind the proof of the Kleitman--Winston graph container lemma.
  We will describe an algorithm which takes as input an orientation $\vec{G} \in \mathcal{D}_{r-1}(G) \cap \mathcal{S}_r(G, \vec{H}, B, X)$ for some graph $G$ on vertex set $A \cup V(H)$, and outputs a bipartite digraph $\vec{C}$ with parts $A$ and $V(H)$ trivially satisfying properties (\ref{item:contains})~and~(\ref{item:few-antiparallel}).
  Along with $\vec{C}$, the algorithm will output a special subgraph $\vec{T} \subset \vec{C}$, called a \emph{fingerprint}, which will be useful for proving property~\eqref{item:few-containers}.

  Recall that $\alpha = 2^6(\log n)/p$. Recall also, from the definition of $r$-frame, that $A \subset [n] \setminus V(H)$, that $B, X \subset V(H)$ are disjoint and that
   \[
     |A| = \alpha/2,\qquad r\alpha\leq |B|\leq \ell\alpha,\qquad \text{and} \qquad |X|\leq (\ell+1-r) \alpha.
   \]
   Set \[ L = |V(H)| - |B| - |X| - \alpha \] and observe that we may assume $L\geq 0$, since otherwise $|V(H)|<|B|+|X|+\alpha\leq 2\ell\alpha$ and therefore it would suffice to take $\mathcal{C}=\{\vec{C}\}$, where $\vec{C}$ is the complete bipartite digraph between $A$ and $V(H)$. Algorithm~\textsc{\ref{alg:cA}}, which will be used to construct the desired family $\mathcal{C}$, is presented below.   We will denote the outputs of $\textsc{Encode}(\vec{G})$ by $\vec{T}(\vec{G})$ and $\vec{C}(\vec{G})$, respectively, and define
  \begin{equation}\label{eq:def_C}
    \mathcal{C} = \left\{\, \vec{C}(\vec{G}) : \vec{G} \in \mathcal{D}_{r-1}(G) \cap \mathcal{S}_r(G, \vec{H}, B, X) \, \right\}.
  \end{equation}
  
  \begin{procedure}[h]
    \SetKwInOut{Input}{Input}
    \SetKwInOut{Output}{Output}
    \Input{An extension $\vec{G}$ of $\vec{H}$ on vertex set $A \cup V(H)$}
    \Output{Digraphs $\vec{T}$ and $\vec{C}$ such that $E(\vec{T}) \subset E_{\vec{G}}(A, V(H)) \subset E(\vec{C})$}
    \ForEach{$a \in A$}{
      $B_0$ $\gets$ $B$ \\
      \For{$i \gets 1 \textrm{ to } L$}{
        pick $v_i \in V(H) \setminus (B \cup X \cup \{v_1, \ldots, v_{i-1}\}$) maximising $\dlr_{(\vec{H}\setminus X)^{r-1}}(\,\cdot\,, B_{i-1})$, with ties broken canonically \label{line:pick_v} \\
        \eIf{$\big|N^+_{(\vec{H}\setminus X)^{r-1}}(v_i,B_{i-1})\big|\geq \frac{1}{2} \big|N_{(\vec{H}\setminus X)^{r-1}}(v_i,B_{i-1})\big|$\label{line:condition_e}}
            {$e_i$ $\gets$ $av_i$\label{line:if_branch_e}}
            {$e_i$ $\gets$ $v_ia$\label{line:else_branch_e}}
        \eIf{$e_i \in E(\vec{G})$\label{line:condition_G}}
           {add the edge $e_i$ to $\vec{T}$ and to $\vec{C}$ \label{line:if_branch_C} \\ $B_{i}$ $\gets$ $B_{i-1} \setminus N_{(\vec{H} \setminus X)^{r-1}}(v_i)$\label{line:if_branch_B}}
           {add the edge $\operatorname{reverse}(e_i)$ to $\vec{C}$, where $\operatorname{reverse}(uw) = wu$ \label{line:else_branch_C} \\ $B_{i}$ $\gets$ $B_{i-1}$\label{line:else_branch_B}}
      }
      \ForEach{$v \in V(H) \setminus \{v_1, \ldots, v_L\}$\label{line:double_edges_loop}}{add the edges $av$ and $va$ to $\vec{C}$\label{line:double_edges}}
    }
    \Return{$\vec{T}$ and $\vec{C}$}
  \caption{Encode(): A container-like algorithm for $\circ_{\ell+2}$-free extensions of $\vec{H}$}\label{alg:cA}
  \end{procedure}

  Observe that, for every extension $\vec{G}$ of $\vec{H}$, it holds that $E(\vec{T}) \subset E_{\vec{G}}(A, V(H))$, since edges are only added to $\vec{T}$ in line \ref{line:if_branch_C}. Moreover, whenever $e_i = uw$ is added to $\vec{C}$ in line \ref{line:if_branch_C}, it holds that $wu \not\in E(\vec{G})$, and therefore $E_{\vec{G}}(A, V(H)) \subset E(\vec{C})$, verifying property \eqref{item:contains}. Similarly, property \eqref{item:few-antiparallel} holds because, for every $a
  \in A$, the loop in line~\ref{line:double_edges_loop} adds $|V(H)|-L\leq 2\ell\alpha$
  pairs of antiparallel edges to $\vec{C}$ by the choice of $L$, and
  the other places in which $\vec{C}$ is modified
  (lines~\ref{line:if_branch_C}~and~\ref{line:else_branch_C}) add
  edges in one direction only.

  Showing that $\mathcal{C}$ satisfies property \eqref{item:few-containers} will require the most amount of work.
  The intuition for it (which will be formalised in Claim~\ref{claim:maxdeg_T}) is easy to describe, however: the if-statement in lines~\ref{line:condition_e}--\ref{line:else_branch_e} determines the orientation of $\{a, v_i\}$ that ``extends the majority of $(r-1)$-paths'' with one endpoint in $B_i$ to $r$-paths.
  If this directed edge is in $\vec{G}$, it is added to $\vec{T}$.
  Therefore, if the input $\vec{G}$ is $(r-1)$-locally-dense, each edge added to $\vec{T}$ extends $\Omega(p^{\frac{\ell-r+2}{\ell}}|B|)$ $(r-1)$-paths of $\vec{H}$ to $r$-paths of $\vec{G}$.
  Since $\vec{G}$ is an $(r, B, X)$-sparse extension of $\vec{H}$, however, there can be at most $O(p^{\frac{\ell-r+1}{\ell}}|B|)$ $r$-paths starting or ending at a given $a \in A$. Therefore, at most $O(p^{-1/\ell})$ edges can be added to $\vec{T}$ in line~\ref{line:if_branch_C} for each $a \in A$. Since $\vec{T}$ determines $\vec{C}$ (Claim~\ref{claim:T_determines_C}), this implies that there are few choices for $\vec{C}$.

  We start the formal proof of \eqref{item:few-containers} by showing that, given an oriented graph $\vec{H}$ and an $r$-frame $(A, H, B, X)$, the output $\vec{T}$ of the algorithm determines the output $\vec{C}$ as well.

  \begin{claim}\label{claim:T_determines_C}
    Let $\vec{G}_1$ and $\vec{G}_2$ be extensions of $\vec{H}$ on vertex set $A \cup V(H)$. If $\vec{T}(\vec{G}_1) = \vec{T}(\vec{G}_2)$, then $\vec{C}(\vec{G}_1) = \vec{C}(\vec{G}_2)$ (even if the underlying graphs $G_1$ and $G_2$ differ).
  \end{claim}

  \begin{proof}
    Line~\ref{line:condition_G} is the only time at which the
    algorithm examines edges outside $V(H)$. Therefore, if
    $\vec{C}(\vec{G}_1) \neq \vec{C}(\vec{G}_2)$, we may suppose without loss of generality there exists $a \in
    A$ and $1 \leq i \leq L$ such that $e_i\in E(\vec{G}_1)$ and $e_i\not \in E(\vec{G}_2)$. If $i$ is minimal with this property, then $e_i$ is
    the same in both executions, but then $e_i$ is added to $\vec{T}(\vec{G}_1)$ and not added to $\vec{T}(\vec{G}_2)$,
    contradicting the fact that $\vec{T}(\vec{G}_1) = \vec{T}(\vec{G}_2)$.
  \end{proof}

  If the extension $\vec{G}$ of $\vec{H}$ given as input to
  $\textsc{Encode}$ is not restricted, then it is possible for the
  graph $\vec{T}$ to have many edges. Roughly speaking, the following
  claim shows this cannot happen when $\vec{G} \in
  \mathcal{D}_{r-1}(G) \cap \mathcal{S}_r(G, \vec{H}, B, X)$. In other
  words, for these digraphs, $\vec{T}$ can be thought of as an
  ``efficient encoding'' of most of $E_{\vec{G}}(A, V(H))$.

  \begin{claim}\label{claim:maxdeg_T}
    Let $\vec{G} \in \mathcal{D}_{r-1}(G) \cap \mathcal{S}_r(G,
    \vec{H}, B, X)$. If $\vec{T} = \vec{T}(\vec{G})$ is the
    fingerprint output by \textsc{Encode} with input $\vec{G}$, then $d_{\vec{T}}(a) \leq 8p^{-1/\ell}$ for every $a \in A$.
    \end{claim}

    \begin{proof}
      Fix $a \in A$, and let $S$ be the set of values of $i \in [L]$ for which lines~\ref{line:if_branch_C}~and~\ref{line:if_branch_B} of the algorithm are executed (that is, for which $e_i \in E(\vec{G})$). Observe that $|S| = d_{\vec{T}}(a)$, since only line~\ref{line:if_branch_C} adds edges to $\vec{T}$. Therefore, our goal is to upper bound $|S|$. We claim that
      \begin{equation}\label{eq:lower_r_paths}
        \frac{1}{2} \sum_{i \in S} \dlr_{(\vec{H} \setminus X)^{r-1}}(v_i, B_{i-1}) \leq \dlr_{(\vec{G} \setminus X)^r}(a,B) \leq p^{(\ell-r+1)/\ell}|B|.
      \end{equation}
      Indeed, the second inequality follows directly from the assumption that $\vec{G} \in \mathcal{S}_r(G, \vec{H}, B, X)$. To see the first inequality, let $S^+ \subset S$ be those values of $i$ for which $d^+_{(\vec{H}\setminus X)^{r-1}}(v_i,B_{i-1})\geq \frac{1}{2} \dlr_{(\vec{H}\setminus X)^{r-1}}(v_i,B_{i-1})$ (see line~\ref{line:condition_e}), let $S^- = S \setminus S^+$, and define
      \[ N_i =
        \begin{cases}
          N^+_{(\vec{H} \setminus X)^{r-1}}(v_i, B_{i-1}) &\text{if }i \in S^+\\
          N^-_{(\vec{H} \setminus X)^{r-1}}(v_i, B_{i-1}) &\text{if }i \in S^-.
        \end{cases}
      \]
      Then $N_i \subset N_{(\vec{G} \setminus X)^r}(a, B)$, since we
      can use the edge $e_i$ to extend an $(r-1)$-path starting or
      ending at $v_i$ to an $r$-path starting or ending at $a$.
      Moreover, since $N_i \subset B_{i-1}$ by definition and $N_{j}
      \cap B_{i-1} = \emptyset$ if $j < i$ by
      line~\ref{line:if_branch_B}, the family $\{N_i : i \in S\}$ is
      pairwise disjoint.  Therefore, $\sum_{i\in S}|N_i| \leq
      \dlr_{(\vec{G} \setminus X)^r}(a,B)$.  Since $|N_i| \geq
      \frac{1}{2} \dlr_{(\vec{H}\setminus X)^{r-1}}(v_i,B_{i-1})$ by
      line~\ref{line:condition_e} (that is, by the condition used to
      decide whether an $i \in S$ is in $S^+$ or $S^-$), we
      obtain~\eqref{eq:lower_r_paths}.

      We will obtain an upper bound on $|S|$ by lower bounding the
      terms in the left-hand side sum of~\eqref{eq:lower_r_paths},
      which will use the assumption that $\vec{G}\in \cD_{r-1}(G)$.
      To do so, recall that $v_i$ was chosen in line~\ref{line:pick_v}
      to maximise $f(x) = \dlr_{(\vec{H}\setminus X)^{r-1}}(x,
      B_{i-1})$ on $V_i \coloneqq V(H) \setminus (B \cup X \cup \{v_1,
      \ldots, v_{i-1}\})$ and note that $|V_i| \geq \alpha$.  Thus, let
      $A'$ be any subset of $V_i$ with $v_i \in A'$ and $|A'| =
      \alpha$ and observe that $\vec{H} \setminus X = \vec{G}
      \setminus (X \cup A)$. Then, we have
      \begin{equation}\label{eq:maximum_avg}
        \dlr_{(\vec{H}\setminus X)^{r-1}}(v_i, B_{i-1}) \geq \frac{1}{|A'|} \cdot e_{(\vec{G} \setminus (X \cup A))^{r-1}}(A', B_{i-1}).
      \end{equation}
      
      We want to use the hypothesis $\vec{G}\in \cD_{r-1}(G)$ to bound the right-hand side of~\eqref{eq:maximum_avg}.
      For that, from Definition~\ref{def:locally_dense}, we need that $|A'| = \alpha$, $(r-1)\alpha\leq |B_{i-1}|\leq \ell\alpha$, and $|X\cup A|\leq (\ell+2-r) \alpha$.
      It suffices to show that $|B_{i-1}| \geq (r-1)\alpha$, since the other bounds hold trivially.
      We observe that $|B_0| = |B|$ and, for every $i \in [L]$,
      \[ |B_{i-1}| - |B_{i}| =
        \begin{cases}
          \dlr_{(\vec{H}\setminus X)^{r-1}}(v_i,B_{i-1})&\text{if }i \in S \\
          0&\text{if }i \in [L] \setminus S,
        \end{cases}
      \] by lines~\ref{line:if_branch_B}~and~\ref{line:else_branch_B}. Therefore, summing and using~\eqref{eq:lower_r_paths}, we obtain
    \[ |B_{L}| \geq |B| - \sum_{i \in S} \dlr_{(\vec{H} \setminus X)^{r-1}}(v_i, B_{i-1}) \geq |B| - 2p^{(\ell-r+1)/\ell}|B|. \]
    Since $p \leq \left(2\ell \right)^{-\ell}$, we have $|B_{i-1}| \geq |B_L| \geq (1-1/\ell)|B|$. In particular, $|B_{i-1}| \geq (r-1)\alpha$, as claimed. Therefore, the condition in Definition~\ref{def:locally_dense} tells us that
    \begin{equation}\label{eq:locally_dense_applied}
      e_{(\vec{G} \setminus (X \cup A))^{r-1}}(A', B_{i-1}) \geq \frac{1}{2} \cdot p^{(\ell-r+2)/\ell}|A'||B_{i-1}|.
    \end{equation} Combining~\eqref{eq:maximum_avg}~and~\eqref{eq:locally_dense_applied}, plugging the resulting bound into~\eqref{eq:lower_r_paths} and using that $|B_{i-1}| \geq |B_L|$, we obtain that
    \[
      \frac{|S|}{4} \cdot p^{(\ell-r+2)/\ell} \cdot |B_{L}| \leq p^{(\ell-r+1)/\ell}|B|,
    \]
    As we have already shown that $|B_L| \geq |B|/2$, we obtain that $d_{\vec{T}}(a) = |S| \leq 8p^{-1/\ell}$, as claimed.
  \end{proof}

  We have already seen that the collection $\mathcal{C}$, defined in~\eqref{eq:def_C}, satisfies the claimed properties \eqref{item:contains} and \eqref{item:few-antiparallel}, and we are now ready to prove \eqref{item:few-containers}.
  Let
  $$
  \mathcal{T} = \left\{\, \vec{T}(\vec{G}) : \vec{G} \in
    \mathcal{D}_{r-1}(G) \cap \mathcal{S}_r(G, \vec{H}, B, X) \,
  \right\} $$ and notice that by Claim \ref{claim:T_determines_C} we
  have $|\mathcal{C}| = |\mathcal{T}|$.
  Moreover, for every $\vec{T} \in \mathcal{T}$, recall that $\vec{T}$ only contains edges incident to $A$ and that, by Claim~\ref{claim:maxdeg_T}, $|N^+_{\vec{T}}(a) \cup N^-_{\vec{T}}(a)| \leq 8p^{-1/\ell}$ for each $a\in A$. By considering the number of ways of choosing these neighbourhoods, we obtain
  $$|\mathcal{T}|\leq \binom{n}{\leq 8p^{-1/\ell}}^{2|A|}\leq
  \exp\left(8 \alpha p^{-1/\ell}\log n\right),$$
  as desired.
\end{proof}

Having proved Proposition~\ref{prop:container-like}, we are now ready to deduce Lemma~\ref{lemma:key}(\ref{lemma:locally-sparse}) from Proposition~\ref{prop:container-like}.

\begin{proof}[Proof of Lemma \ref{lemma:key}(\ref{lemma:locally-sparse})]

  Let $G = G(A,H,p)$, recalling that, by definition, we have $V(G) = A \cup V(H)$, $G[V(H)]=H$ and each element of $\binom{V(G)}{2}\setminus \binom{V(H)}{2}$ is an edge of $G$ with probability $p$ independently at random. Let $\mathcal{O} = \cD_{r-1}(G) \cap \cS_r(G,\vec{H},B,X)$, and recall that our goal is to bound the expected size of $\mathcal{O}$.

  Let $\mathcal{C}$ be given by Proposition~\ref{prop:container-like} and say an orientation $\vec{G} \in \mathcal{O}$ is \emph{compatible} with a container $\vec{C} \in \mathcal{C}$ if $E_{\vec{G}}(A, V(H)) \subset E(\vec{C})$, and let $D(\vec{C})$ be the set of pairs $\{u, v\}$ for which $\{uv, vu\} \subset E(\vec{C})$. Given a container $\vec{C} \in \mathcal{C}$, let
  \[ N(\vec{C}) = |E(G) \cap (D(\vec{C}) \cup (A\times A))| \]
  and observe that $2^{N(\vec{C})}$ is an upper bound on the number of orientations $\vec{G} \in \mathcal{O}$ which are compatible with $\vec{C}$. Moreover, by Proposition~\ref{prop:container-like}(\ref{item:contains}), every element of $\mathcal{O}$ is compatible with at least one element of $\mathcal{C}$, and therefore
  \[ |\mathcal{O}| \leq \sum_{\vec{C} \in \mathcal{C}} 2^{N(\vec{C})}. \]
  We can take expectations on both sides and apply Lemma~\ref{lemma:chernoff} to obtain
  \[
    \Ex\big[|\mathcal{O}|\big]\leq \sum_{\vec{C} \in \mathcal{C}} e^{2p\cdot (|D(\vec{C})| + |A\times A|)} \leq \exp\left( 8 \alpha p^{-1/\ell}\log n+3\ell p\alpha^2\right),
  \]
  where the last inequality uses
  Proposition~\ref{prop:container-like}(\ref{item:few-antiparallel})--(\ref{item:few-containers})
  to bound $|D(\vec{C})|$ and $|\mathcal{C}|$, respectively. Since
  $p\alpha= 2^6 \log n$ and $p\leq (2^8\ell)^{-\ell}$, we have $\ell p \alpha \leq p^{-1/\ell} (\log n)/4$. Therefore,
  $$\Ex\big[|\mathcal{O}|\big]\leq \exp\left( 9 \alpha
    p^{-1/\ell}\log n \right).$$
   This finishes the proof.
\end{proof}

\section{Extending locally-dense orientations}\label{sec:locally-dense}

In this section we prove Lemma~\ref{lemma:key}(\ref{lemma:locally-dense}), which bounds the
expected number of $\ell$-locally-dense extensions of an oriented graph $\vec{H}$.
For that we use the \emph{graph container lemma}, implicit in the work of Kleitman and Winston~\cite{KlWi82} and explicitly stated by Sapozhenko~\cite{Sa01}.
The version we use can be found on a survey of Samotij~\cite{SamotijIndep}.

\begin{lemma}[{\cite{SamotijIndep}*{Lemmas 1 and 2}}]
  \label{lemma:container}
  Let $G$ be a graph on $n$ vertices, an integer $q$ and reals $0 \leq \beta \leq 1$ and $R$ such that $R \geq e^{-\beta q} n$.
  Suppose that every set $U \subset V(G)$ with $|U| \geq R$ satisfies
  \[
    e(G[U]) \geq \beta \binom{|U|}{2}.
  \]
  Then there exists a collection $\mathcal{C} \subset \mathcal{P}(V(G))$ such that:
  \begin{enumerate}[$(i)$]
    \item Every independent set of $G$ is contained in some $C \in \mathcal{C}$,
    \item $|C| \leq R+q$ for every $C \in \mathcal{C}$,
    \item $|\mathcal{C}| \leq \sum_{i=0}^q \binom{n}{i}$.
    \end{enumerate}
  \end{lemma}

The following simple averaging result will be useful for applying the graph container lemma in the proof of Lemma~\ref{lemma:key}~\eqref{lemma:locally-dense}.

\begin{lemma}\label{lemma:edgesaveraging}
  Let $\vec{G}$ be an $\ell$-locally-dense oriented graph.
  For every pair of disjoint sets $U,~X\subset V(G)$ with $|U| \geq (\ell+1) \alpha$ and $|X|= \alpha$, the graph $\vec{H}=\vec{G}\setminus X$ satisfies $e(\vec{H}^\ell[U]) \geq p^{1/\ell}\binom{|U|}{2}/(2\ell+2)$.
\end{lemma}

\begin{proof}
  Let $W$ be a random subset of $U$ of size $(\ell+1)\alpha$.
  By splitting $W$ arbitrarily into two sets of size $\alpha$ and $\ell\alpha$ and applying the hypothesis that $\vec{G}$ is $\ell$-locally-dense, we see that
  \[
    e(\vec{H}^\ell[W]) \geq \frac{p^{1/\ell}\cdot\ell\alpha^2}{2} \geq \frac{\ell p^{1/\ell}}{(\ell+1)^2} \cdot \binom{|W|}{2} \geq \frac{p^{1/\ell}}{2(\ell+1)} \cdot \binom{|W|}{2}.
  \]
  On the other hand, the probability that an edge of $\vec{H}[U]$ is present in $\vec{H}[W]$ equals $\binom{|W|}{2} / \binom{|U|}{2}$. Therefore, $\mathbb{E}[e(\vec{H}[W])] =  e(\vec{H}[U]) \cdot \binom{|W|}{2} / \binom{|U|}{2}$, and the result follows.
\end{proof}

We are now ready to estimate the expected number of $\ell$-locally
dense extensions of an oriented graph $\vec{H}$. Recall that we need
to show that
$$
\Ex\left[\big|\cD_\ell(G, \vec{H})\big|\right]\leq
\exp\left(\alpha(\ell+2) p^{-\frac{1}{\ell}}(\log n)^2\right).
$$

\begin{proof}[Proof of Lemma~\ref{lemma:key}~\eqref{lemma:locally-dense}]
  Given the random graph $G=G(A,H,p)$ and an orientation $\vec{H}$, we
  want to find the expected size of $\cD_\ell(G, \vec{H})$, that is,
  the expected number of $\circ_{\ell+2}$-free, $\ell$-locally-dense
  extensions of $\vec{H}$.  We will follow the proof strategy
  from~\cite{CoKoMoMo20}.  For each orientation $\vec{G}$, let
  $\vec{T} = \vec{T}(\vec{G})$ be a minimal subset of $E(\vec{G})
  \setminus E(\vec{H})$ such that $\vec{G}$ is the unique
  $\circ_{\ell+2}$-free orientation of $G$ containing $\vec{T} \cup
  E(\vec{H})$. For $a \in A$, let
  \[ \vec{T}^+_a = \{ v \in V(H) : av \in \vec{T} \}. \]
  We claim that $\vec{T}^+_a$ is an independent set in
  $\vec{H}^\ell$. Indeed, suppose there are $x, y \in
  \vec{T}^+_a$ such that $xy \in E(\vec{H}^\ell)$. Let $\vec{T}' =
  \vec{T} \setminus \{ay\}$ and observe that every orientation
  containing $\vec{T}' \cup \{ya\} \cup \vec{H}$ contains a
  $\circ_{\ell+2}$ of the form $ax\vec{P}ya$ for some path $\vec{P}
  \subset \vec{H}$ of length $\ell$. Therefore, any $\circ_{\ell+2}$-free
  orientation of $\vec{G}$ containing $\vec{T}' \cup \vec{H}$ also
  contains the edge $ay$, contradicting the minimality of $\vec{T}$.
  By symmetry, the sets $\vec{T}^-_a$, defined analogously for $a \in A$, are also independent sets in $\vec{H}^\ell$. With this in mind, set
  \[ \mathcal{T}_1 = \big\{ \, \vec{T}(\vec{G}) \setminus (A \times A) : \vec{G} \in \mathcal{D}_\ell(G, \vec{H}) \, \big\}. \]
To bound $|\mathcal{T}_1|$ using the graph container lemma (Lemma~\ref{lemma:container}), set
  \[
    \beta = \frac{p^{1/\ell}}{2(\ell+1)},\qquad q = \frac{2(\ell+1) \log n}{p^{1/\ell}},\qquad R = (\ell+1)\alpha.
  \]
  Using Lemma \ref{lemma:edgesaveraging}, one can check that these constants satisfy the conditions of Lemma~\ref{lemma:container}, from which we obtain a family $\mathcal{C}$ of containers for the independent sets of $\vec{H}^\ell$. Therefore, since $T^+_a$ and $T^-_a$ are independent sets in $\vec{H}^\ell$ contained in $N_G(a)$ for each $a \in A$, we obtain
  \begin{equation}\label{eq:num_indep_sets}
    |\mathcal{T}_1| \leq \prod_{a \in A} \Big(\sum_{C \in \cC} 2^{|C \cap N_G(a)|}\Big)^2 \leq \prod_{a \in A} \Big( |\mathcal{C}| \cdot \sum_{C \in \cC} 4^{|C \cap N_G(a)|}\Big),
  \end{equation}
  where the last inequality follows by convexity.
  On the other hand, it is clear that the set $\mathcal{T}_2 = \{ \vec{T}(\vec{G}) \cap (A \times A) : \vec{G} \in \mathcal{D}_\ell(G, \vec{H}) \}$ satisfies
  \begin{equation}\label{eq:num_A}
    |\mathcal{T}_2| \leq 3^{|E(G) \cap (A \times A)|}.
  \end{equation}
  Since $\vec{T}(\vec{G})$ uniquely determines an orientation $\vec{G} \in \mathcal{D}_\ell(G, \vec{H})$ by definition, we can use linearity of expectation and the fact that the exponents in \eqref{eq:num_indep_sets}~and~\eqref{eq:num_A} depend on pairwise disjoint sets of edges of $G$ to obtain
  \begin{equation}\label{eq:dense_almost_done}
    \mathbb{E}\big[|\mathcal{D}_\ell(G, \vec{H})|\big] \leq \mathbb{E}\big[3^{|E(G) \cap (A \times A)|}\big] \cdot  \prod_{a \in A} \Big(|\mathcal{C}| \cdot \sum_{C \in \cC} \mathbb{E}\Big[ 4^{|C \cap N_G(a)|} \Big] \Big)
  \end{equation}
  By Lemma~\ref{lemma:container}, we have
  \[
    |\mathcal{C}| \leq \sum_{i=0}^q \binom{n}{i} \leq \left(\frac{en}{q}\right)^q \leq e^{q \log n}\, .
  \]
  Moreover, since  $|C| \leq R+q \leq 2R$ for every $C \in \mathcal{C}$, we can apply Lemma~\ref{lemma:chernoff} and obtain $$\mathbb{E}\big[3^{|E(G) \cap (A \times A)|}\big]\leq e^{3p|A|^2} \quad \text{and} \quad \mathbb{E}\Big[ 4^{|C \cap N_G(a)|}\Big]\leq e^{8pR}\, .$$
  Plugging these bounds in~\eqref{eq:dense_almost_done} and recalling that $|A| = \alpha/2$, $R=(\ell+1)\alpha$ and $q=2(\ell+1)p^{-1/\ell} \log n$, we have
  \begin{align*}
    \mathbb{E}\big[|\mathcal{D}_\ell(G, \vec{H})|\big] &\leq \exp\Big(3p|A|^2 + |A| (q \log n + 8pR)\Big) \\
                                                       &\leq \exp\Big(p\alpha^2 + \alpha (\ell+1)p^{-1/\ell} (\log n)^2 + 4p(\ell+1)\alpha^2\Big) \\
                                                       &\leq \exp\Big(\alpha(\ell + 2) p^{-1/\ell} (\log n)^2\Big),
  \end{align*}  
  where in the last inequality we used that $\ell p \alpha = o(p^{-1/\ell} (\log n)^2)$, which follows from $p\alpha=2^6 \log n$ and $p\leq \left(2^7 \ell \right)^{-\ell}$. This finishes the proof.
\end{proof}

\section{Concluding Remarks}

This paper fits into the effort of understanding the number of $\vec{H}$-free orientations of $G(n,p)$, and there is still much to be understood about this problem.
It is still an open problem to provide good estimates on $\log
D(G(n,p),\vec{H})$ when $\vec{H}$ is a strongly connected tournament
with at least $4$ vertices.
\begin{conj}\label{conj:tournaments}
  Let $\vec{H}$ be a strongly connected tournament with $k:=v(H)\geq 4$ and let $p\gg n^{-2/(k+1)}$.
  Then, with high probability, \[\log D(G(n,p),\vec{H})= \tilde{\Theta}\left(\frac{n}{p^{(k-1)/2}}\right).\]
\end{conj}

We remark that a more general version of Conjecture~\ref{conj:tournaments} appeared in~\cite{CoKoMoMo20}.
Coming back to the case studied in the present paper, there are also several open problems when $\vec{H}$ is a directed cycle.
We believe Theorem~\ref{thm:cycles} can be tightened as in the following.
\begin{conj}
  If $p\gg n^{-1+1/(k-1)}$ then, with high probability, \[\log D(G(n,p),\circ_k)= \Theta\left(\frac{n}{p^{1/(k-2)}}+n\log n\right).\]
\end{conj}

The lower bound for this conjecture was proved in \cite{CoKoMoMo20}, but the upper bound is not known even in the case $k = 3$.
It would also be interesting to understand the behaviour of $\log D(G(n,p),\circ_k)$ as $k$ grows with $n$.
For instance, it is not clear to us what should happen when $k\gg \log n$.

\section{Acknowledgements}
The research that led to this paper started at WoPOCA 2019, and we thank the workshop organisers for a productive working environment. We would also like to thank Rob Morris for helpful comments on previous drafts of this paper, as well as the anonymous referees for their careful reading and valuable comments.

\bibliography{bibliografia}

\begin{bibdiv}
\begin{biblist}

\bib{AlKoMoPa14}{article}{
      author={Allen, P.},
      author={Kohayakawa, Y.},
      author={Mota, G.~O.},
      author={Parente, R.~F.},
       title={On the number of orientations of random graphs with no directed
  cycles of a given length},
        date={2014},
        ISSN={1077-8926},
     journal={Electron. J. Combin.},
      volume={21},
      number={1},
       pages={Paper 1.52, 13},
      review={\MR{3192383}},
}

\bib{AlYu06}{article}{
      author={Alon, Noga},
      author={Yuster, Raphael},
       title={The number of orientations having no fixed tournament},
        date={2006},
        ISSN={0209-9683},
     journal={Combinatorica},
      volume={26},
      number={1},
       pages={1\ndash 16},
         url={http://dx.doi.org/10.1007/s00493-006-0001-6},
      review={\MR{MR2201280 (2006j:05095)}},
}

\bib{ArBoMo20+}{article}{
      author={Araújo, Pedro},
      author={Botler, Fábio},
      author={Mota, Guilherme~Oliveira},
       title={{Counting graph orientations with no directed triangles}},
        date={2020},
      eprint={2005.13091},
}

\bib{BoHoMo22}{article}{
      author={Botler, F{\'a}bio},
      author={Hoppen, Carlos},
      author={Mota, Guilherme~Oliveira},
       title={Counting orientations of graphs with no strongly connected
  tournaments},
        date={2022},
        ISSN={0012-365X},
     journal={Discrete Math.},
      volume={345},
      number={12},
        note={Id 113024},
}

\bib{BuJaSu23}{article}{
      author={Buci{\'c}, Matija},
      author={Janzer, Oliver},
      author={Sudakov, Benny},
       title={Counting {{\(H\)}}-free orientations of graphs},
        date={2023},
        ISSN={0305-0041},
     journal={Math. Proc. Camb. Philos. Soc.},
      volume={174},
      number={1},
       pages={79\ndash 95},
}

\bib{CoKoMoMo20}{article}{
      author={Collares, Maurício},
      author={Kohayakawa, Yoshiharu},
      author={Morris, Robert},
      author={Mota, Guilherme~O.},
       title={Counting restricted orientations of random graphs},
        date={2020},
        ISSN={1098-2418},
     journal={Random Structures \& Algorithms},
      volume={56},
      number={4},
       pages={1016–1030},
         url={http://dx.doi.org/10.1002/rsa.20904},
}

\bib{Er74}{inproceedings}{
      author={Erd{\H{o}}s, P.},
       title={Some new applications of probability methods to combinatorial
  analysis and graph theory},
        date={1974},
   booktitle={Proceedings of the {F}ifth {S}outheastern {C}onference on
  {C}ombinatorics, {G}raph {T}heory and {C}omputing ({F}lorida {A}tlantic
  {U}niv., {B}oca {R}aton, {F}l, 1974)},
   publisher={Utilitas Math.},
     address={Winnipeg, Man.},
       pages={39\ndash 51. Congressus Numerantium, No.~X},
      review={\MR{MR0364020 (51 \#275)}},
}

\bib{HararyMoser1966}{article}{
      author={Harary, Frank},
      author={Moser, Leo},
       title={The theory of round robin tournaments},
        date={1966},
        ISSN={0002-9890,1930-0972},
     journal={Amer. Math. Monthly},
      volume={73},
       pages={231\ndash 246},
         url={https://doi.org/10.2307/2315334},
      review={\MR{197347}},
}

\bib{KlWi82}{article}{
      author={Kleitman, Daniel~J},
      author={Winston, Kenneth~J},
       title={On the number of graphs without 4-cycles},
        date={1982},
     journal={Discrete Mathematics},
      volume={41},
      number={2},
       pages={167\ndash 172},
}

\bib{SamotijIndep}{article}{
      author={Samotij, Wojciech},
       title={Counting independent sets in graphs},
        date={2015},
        ISSN={0195-6698},
     journal={European Journal of Combinatorics},
      volume={48},
       pages={5–18},
         url={http://dx.doi.org/10.1016/j.ejc.2015.02.005},
}

\bib{Sa01}{article}{
      author={Sapozhenko, Aleksandr~Antonovich},
       title={On the number of independent sets in extenders},
        date={2001},
     journal={Diskretnaya Matematika},
      volume={13},
      number={1},
       pages={56\ndash 62},
}

\end{biblist}
\end{bibdiv}
\end{document}